\documentclass[11pt]{article}

\usepackage{subfigure}
\usepackage{color}
\usepackage[cmex10]{amsmath}
\usepackage{mathrsfs,amssymb,amsmath}
\usepackage{cases}
\usepackage{graphicx}
\usepackage{caption}

\usepackage{epstopdf}
\usepackage{hyperref}
\usepackage{float}
\usepackage{placeins}

\def\disp{\displaystyle}

\newcommand{\R}{{\mathbb R}}

\def\dref#1{(\ref{#1})}
\def\crr{\cr\noalign{\vskip-0.0mm}}

\usepackage{algorithm} %format of the algorithm
\usepackage{algorithmic} %format of the algorithm

\newtheorem{corollary}{Corollary}[section]
\newtheorem{remark}{Remark}[section]
\newtheorem{lemma}{Lemma}[section]

\newtheorem{proposition}{Proposition}

\newtheorem{theorem}{Theorem}[section]
\newtheorem{definition}{Definition}[section]
\newtheorem{notation}{Notation}[section]

\newenvironment{proof}{{\bf {Proof.}}}{\hfill $\square$}
\allowdisplaybreaks[4]

\numberwithin{equation}{section}

\def\dref#1{(\ref{#1})}

\def\pt{\partial}

\def\ra{\rightarrow}

\def\s{\subseteq}

\def\e{\varepsilon}

\def\vp{\varphi}
\def\lg{\langle}

\def\rg{\rangle}

\def\es{\emptyset}

\def\bf{\textbf}
\def\pt{\partial}

\def\Om{\Omega}
\def\la{\lambda}
\def\al{\alpha}
\def\be{\beta}
\def\de{\delta}
\def\ga{\gamma}

\def\Ga{\Gamma}

\def\ts{\times}

\def\iy{\infty}

\def\f{\frac}

\def\Lra{\Leftrightarrow}

\def\Span{{\rm{span}}}

\def\df{\mathrm d}

\def\wh{\widehat}

\def\esssup{\operatorname*{ess\ \! sup}}

\def\hra{\hookrightarrow}

\def\mcO{\mathcal{O}}
\def\mcA{\mathcal{A}}

\def\mcD{\mathcal{D}}

	\DeclareMathOperator{\Div}{div}

	\DeclareMathOperator{\dist}{dist}

	\DeclareMathOperator{\supp}{{supp}}

	\newcommand{\N}{\mathbb N}

\begin{document}

\title{{\bf   {\bf        Shape Design for Degenerate Parabolic Equations with Degenerate Boundaries and Its Application to Boundary Observability}}\footnote{\small This work was carried out with the support of the
National Natural Science Foundation of China under grant  nos. 12131008 and U23B2033, and
National Key R\&D Program of China under grant no. 2024YFA1013101.}}

\author{ Dong-Hui Yang$^{a}$, Bao-Zhu Guo$^{b}$\footnote{\small
The corresponding author. Email: bzguo@is.ac.cn}
\\
$^a${\it School of Mathematics and Statistics, Central South University}\\
			{\it Changsha 410075, P.R.China}\\
$^b${\it Academy of Mathematics and Systems Science, Academia Sinica, Beijing 100190, China}}

%\date{September 29, 2021}
\date{}

\maketitle{}

\begin{abstract}
     In this study, we firstly establish the well-posedness of a degenerate parabolic equation under Dirichlet boundary conditions. Following this, we introduce a shape design problem, which acts as a framework for approximating the degenerate parabolic equation through a series of uniformly parabolic equations. Finally, as a tangible application of this shape design approach, we deduce a boundary observability inequality associated with the degenerate parabolic equation.

\vspace{0.3cm}

	
\vspace{0.3cm}
		

	\end{abstract}

\section{Introduction}
	
	   Degenerate partial differential equations have been a central focus of substantial research in the existing literature. For in-depth explorations, we  refer to \cite{Cavalheiro,CS2,CS3,Fabes,FF,Guo,Trudinger,Yang,Wu1}. When dealing with these types of equations, it is inherently logical to analyze their solutions within the framework of weighted Sobolev spaces; relevant references include \cite{GC,Heinonen}. Moreover, degenerate partial differential equations can be effectively approximated by sequences of nondegenerate problems, with further insights provided in \cite{Cavalheiro,FF}.

Shape design problems have also been thoroughly investigated, with relevant studies such as \cite{Buttazzo,Chenais,Greco,Guo2,Guo1,Guo,He,Henrot,Privat,Tiba,Wang,Yang}. However, in the realm of degenerate parabolic equations, the research remains relatively sparse; additional information can be found in \cite{Greco,Guo,Yang}.

The controllability and observability of partial differential equations have been extensively studied, as demonstrated by works such as \cite{Beauchard1,Chen,Coron,FG,Fursikov,Lions,Lu,Lebeau,Miller,Weiss,Zuazua}. It is a well-established fact that controllability and observability are equivalent concepts, a relationship established through the Hilbert Uniqueness Method introduced by Lions \cite{Lions}.

For degenerate partial differential equations, particularly in one- and two-dimensional settings, numerous results have been obtained; for further details, consult \cite{Beauchard,Buffe1,new1,Yang1}. Conversely, in higher-dimensional cases, only a limited number of studies have been conducted \cite{Guo,Wu,Yang}. One of the primary challenges in higher dimensions stems from the absence of appropriate integration-by-parts formulas.

In \cite{Wu}, the authors employed a cut-off technique to derive general controllability and observability results. In \cite{Yang1}, the method of separation of variables was used to establish controllability in the two-dimensional context. In \cite{Guo}, shape design methods were pioneered to investigate controllability problems for one-dimensional degenerate equations.

In this study, we utilize a shape design methodology to explore the boundary controllability of higher-dimensional degenerate parabolic equations. The advantage of this approach lies in its ability to approximate the degenerate parabolic equation through a series of uniformly parabolic equations, along with uniform observability estimates. Subsequently, the observability of the degenerate equation is achieved through this approximation process. Notably, the constants within these estimates must exhibit uniformity. In this paper, we demonstrate the uniformity of these constants, marking another significant contribution of our research. It is pertinent to note that shape design methods share conceptual similarities with cut-off techniques. However, while the cut-off technique addresses the issue within a single evolving system, shape design methods generally encompass a collection of distinct systems. Finally, we emphasize that the method of separation of variables is not applicable to the context examined in this paper.

	   The problem addressed in this paper is characterized by the following high-dimensional degenerate parabolic equation:
	\begin{equation}\label{01.08.1}
		\begin{cases}
			\pt_{t}y-\Div (A\nabla y)= f, &\mbox{in } Q,\\
			y=0, &\mbox{on }\pt Q, \\
			y(0)=y^0, &\mbox{in } \Om,
		\end{cases}
	\end{equation}
	  where $\Om\s \R^{N-1}\ts (0,+\iy)$   is a convex domain satisfying
$ \Ga_N^0=\pt\Om\cap (\R^{N-1}\ts \{0\})\neq \es$,  and $0\in\R^N$  is an interior point of
 $\Ga_N^0$.  Additionally, $\pt\Om\in C^3$,    $Q=\Om\ts (0,T)$ for $T>0$, and
\begin{equation}
A = \mathrm{diag}(\underbrace{1,\cdots, 1}_{N - 1}, x_N^\alpha), \quad \text{with } \alpha \in (0, 1),
\end{equation}
and $f\in L^2(Q)$, $y^0\in L^2(\Om)$. We define
\begin{equation}\label{02.05.1}
\mathcal{A}y = -\mathrm{div}(A\nabla y), \quad D(\mathcal{A}) = H_0^1(\Omega; w) \cap H^2(\Omega; w).
\end{equation}

  To derive a boundary observability inequality for degenerate parabolic equations, it is essential to enhance the regularity of the normal derivative
$\frac{\partial y}{\partial \nu}$. However, obtaining such regularity on the degenerate part of the boundary presents significant challenges. Fortunately, it is possible to establish the regularity of  $\frac{\partial y}{\partial \nu}$
 on portions of the boundary that are away from the degeneracy. Boundary controls acting on these nondegenerate regions are sufficient to drive the system to zero, thereby achieving null controllability.

We highlight that in \cite{Cannarsa1}, the authors investigated an intriguing two-dimensional degenerate parabolic equation. To derive controllability results, they introduced specific structural assumptions (\cite[Hypothesis 2.2, Chapter 2.1.2, p. 12]{Cannarsa1}). Notably, the coefficient matrix is required to degenerate at the first eigenvalue across the entire boundary. Under this premise, integration by parts remains applicable, which is pivotal in their analysis.

However, this assumption is inadequate for general degenerate parabolic equations, such as \eqref{01.08.1}. In many scenarios, degeneracy arises solely on a segment of the boundary, rather than the entire boundary. In such cases, integration by parts, exemplified by
$(P_{12}z,P_{21}z)_{L^2(Q)}$  in Section \ref{S4} of this paper, generally fails, presenting a fundamental challenge.

One potential strategy to address this issue is to enhance the regularity of weak solutions. Yet, to the best of our knowledge, when degeneracy occurs only on part of the boundary, achieving such a regularity improvement is typically unfeasible. This challenge is also evident in equation \eqref{01.08.1} when  $\alpha \in [1,2)$.  In this scenario, establishing controllability seems unattainable, as the boundary integral
$\sum_{i=1}^{N-1} \int_{(\pt\Om\setminus \Ga^+)\ts (0,T)} x_N\left(\frac{\partial y}{\partial x_i}\right)^2 \, dS \, dt= 0$
is ill-defined. Obtaining this term would necessitate higher solution regularity on
 $\pt\Om\setminus \Ga^+$, which is generally unavailable. The paper \cite{Cannarsa1} resolves this issue only under very specific conditions. Given that
$\sum_{i=1}^{N-1} \int_{(\partial\Omega \setminus\Gamma^+)\times(0,T)} x_N\left(\frac{\partial y}{\partial x_i}\right)^2 \, dS \, dt \ge 0$
is self-evident, a natural approach to overcome this difficulty would be to impose control on the entire boundary. However, boundary control acting on the whole boundary is essentially trivial and thus lacks practical significance.

Another two-dimensional degenerate parabolic equation, set on a rectangular domain, was explored in \cite{Araruna}, where two sides of the boundary exhibit degeneracy. Similar to \cite{Cannarsa1}, they constructed appropriate solution spaces to facilitate integration by parts. In contrast, in \cite{Yang1}, we circumvented this challenge by employing the method of separation of variables.

An alternative and more adaptable approach is the shape design method adopted in this paper. This method eliminates the need for integration by parts on the degenerate part of the boundary. Furthermore, it does not rely on specific geometric assumptions about the domain in the case of Dirichlet boundary conditions. Nevertheless, for Neumann boundary conditions, this issue remains unresolved.

	 For convenience, we introduce the following notation, which will be used throughout the paper.

\begin{notation}\label{01.27.N1}
Denote $N=\{0,1,2,\cdots\}$ and $\N^*=\{1,2,\cdots\}$, and  $M=\sup_{x\in\Om}|x|+1$, and $\{x_N<r\}=\{x=(x',x_N)\in\R^N\colon x_N<r\}$ for some $r\in\R$ with $x'=(x_1,\cdots, x_{N-1})$, and $e_N=(0,\cdots,0,1)\in\R^N$.

It is evident that  $\nu\cdot e_N=-1$ on $\Ga_N^0$. Given $\pt\Om\in C^3$, there exists $\de_0>0$ such that, for all $\de\in (0,4\de_0)$,
 there is at least one $C^3$ convex domain  $\Om_\de$
 satisfying
\begin{equation}\label{01.26.4}
{x = (x', x_N) \in \Omega \colon x_N > 2\delta} \subset \Omega_\delta \subset {x = (x', x_N) \in \Omega \colon x_N < \delta},
\end{equation}
and
\begin{equation}\label{01.27.2}
\Omega_{\delta_1} \subset \Omega_{\delta_2}, \text{ for all } 0 < \delta_2 < \delta_1 < \delta_0,
\end{equation}
and $\nu(x)\cdot e_N\leq 0$ for all $x\in \pt\Om_\de\cap \Om$.  Denote
\begin{equation}
\Gamma^+ = {x \in \partial \Omega \colon A(x)\nu(x) \cdot e_N > 0}.
\end{equation}
It is clear that $\Gamma^+$  depends only  on
 $\Omega$; however, it does not depend on   $\delta \in (0, \delta_0)$.
For all $\de\in (0,\de_0)$, denote
\begin{equation}\label{01.27.4}
 \mcO(\Ga^+,\de)=\left\{x\in\Om\colon \dist(x,\Ga^+)<\de\right\}, \mbox{ with } \dist(x,\Ga^+)=\inf_{z\in \Ga^+}|x-z|.
\end{equation}
Then, for each $x=(x',x_N)\in \mcO(\Ga^+,\de)$, we have  $x_N>\de_0$, and $\mcO(\Ga^+,\de_0)\s \Om_{\de_0}$. Finally, the symbol
$C$
 shall represent a constant; however, its value may vary across different contexts.
\end{notation}

 The main results of this paper are outlined as follows.

\begin{theorem}\label{01.11.T3}
	 Let $\al\in (0,1)$ and $y^0\in C_0^\iy(\Om)$ and $f\in L^2(Q)$. Let
$y(\cdot,\cdot)$
 be the weak solution of \eqref{01.11.14} with respect to $(y^0,f)$, and $y_\de(\cdot,\cdot) (0<\de< \min\{\de_0, \dist(\supp y^0, \Ga_N^0)\})$
  be the weak solution of \eqref{01.11.4} with respect to $(y^0|_{\Om_\de}, f|_{Q_\de})$. Then,
	\begin{equation}\label{01.11.18}
		\begin{split}
			Ey_\de
			&\ra y \mbox{ weakly in } L^2(0,T; H_0^1(\Om;w)), \\
			\pt_t Ey_\de
			&\ra \pt_ty \mbox{ weakly in } L^2(Q),\\
			Ey_\de
			&\ra y \mbox{ strongly in } L^2(Q).
		\end{split}
	\end{equation}
	 Moreover, if we additionally assume $f=0$, then
	\begin{equation}\label{01.12.5}
		\begin{split}
			Ey_\de(T)
			&\ra y(T) \mbox{ strongly in } L^2(\Om), \\
			\f{\pt Ey_\de}{\pt \nu}
			=\f{\pt y_\de}{\pt \nu}
			&\ra \f{\pt y}{\pt \nu} \mbox{ strongly in } L^2(0,T; L^2(\Ga^+)).
		\end{split}
	\end{equation}
\end{theorem}

\begin{theorem}\label{01.27.T2}
	 Let $\al\in (0,1)$ and $y^T\in L^2(\Om)$.
  Let $y(\cdot,\cdot)$
 be the weak solution of \eqref{01.28.1} with respect to $(y^0,f=0)$.  Then,
	\begin{equation*}
		\int_{\Om}y^2(0)\df x\leq C\iint_{\Ga^+\ts (0,T)}\left|\f{\pt y}{\pt \nu}\right|^2\df S\df t,
	\end{equation*}
	where the constant $C>0$ depends only on $\al, T$ and $\Om_{\de_0}$.
\end{theorem}

This paper is organized  as follows. In Section \ref{S2}, we introduce the functional framework and establish the existence of weak solutions to the degenerate partial differential equations. In Section \ref{S3}, we investigate a shape design problem, through which the degenerate parabolic equation is approximated by a sequence of uniformly parabolic equations with appropriately chosen initial data. In Section \ref{S4}, we derive Carleman estimates for the approximating uniformly parabolic equations, with constants that are uniform with respect to the parameter
$(0,\delta_0)$.  Finally, in Section \ref{S5}, we obtain a boundary observability inequality for the degenerate parabolic equations.

\section{Solution spaces}\label{S2}

 In this section, we present the suitable solution spaces for equation \eqref{01.08.1} and subsequently demonstrate the existence of a weak solution to this equation.

Define
\begin{equation*}
	H^1(\Om;w)=\left\{u\in L^2(\Om)\colon \int_\Om \nabla u\cdot A\nabla u\df x<+\iy\right\},
\end{equation*}
where $\nabla u=(\f{\pt u}{\pt x_1}, \cdots, \f{\pt u}{\pt x_N})$
represents the gradient of
$u$. The inner product and norm on  $H^1(\Om;w)$ are defined  as
\begin{equation*}
	(u,v)_{H^1(\Om;w)}=\int_\Om uv\df x+\int_\Om \nabla u\cdot A\nabla v\df x, \quad \|u\|_{H^1(\Om;w)}=(u,u)_{H^1(\Om;w)}^\f{1}{2},
\end{equation*}
respectively.
Let $\mcD(\Om)=C_0^\iy(\Om)$, and
\begin{equation*}
	H_0^1(\Om;w)=\mbox{the closure of $\mcD(\Om)$ in $H^1(\Om;w)$}.
\end{equation*}
We denote $H^{-1}(\Om;w)$  as the dual space of $H_0^1(\Om;w)$.

Define
\begin{equation*}
	H^2(\Om;w)=\left\{u\in H^1(\Om;w)\colon \mcA u\in L^2(\Om)\right\},
\end{equation*}
where $\mcA$ is given by \dref{02.05.1}.
 The inner product and the norm on $H^2(\Om;w)$ are defined as
\begin{equation*}
	(u,v)_{H^2(\Om)}=(u,v)_{H^1(\Om;w)}+\int_\Om \left(\mcA u\right)\left(\mcA v\right)\df x, \quad \|u\|_{H^2(\Om;w)}=(u,u)_{H^2(\Om;w)}^\f{1}{2},
\end{equation*}
respectively.

The following Lemma \ref{Le.g1.1} is a well-established fact  as referenced in \cite{GC,Heinonen}.
\begin{lemma}\label{Le.g1.1}
	The spaces
	\begin{equation*}\label{12.08.1}
		  (H^1(\Om;w),(\cdot,\cdot)_{H^1(\Om;w)}),\quad (H^2(\Om;w), (\cdot,\cdot)_{H^2(\Om;w)})
	\end{equation*}
	are Hilbert spaces.
\end{lemma}

 The following Hardy inequality will be utilized in the subsequent sections.

\begin{lemma}[Hardy's inequality]\label{11.29.L1}
	For all $u\in H_0^1(\Om;w)$, the following inequality holds
	\begin{equation*}
		\int_\Om x_N^{\al-2}u^2\df x\leq C\int_\Om x_N^\al \left|\pt_{x_N} u\right|^2\df x,
	\end{equation*}
	where the constant $C>0$ depends only on $\al$.
\end{lemma}

\begin{proof}
	By employing a density argument, it suffices to prove the case for $u\in C_0^\iy(\Om)$. Naturally, $u\in C_0^\iy(\R^N)$. Let $\be\in (\al,1)$. We compute
\begin{eqnarray*}
	\int_\Om x_N^{\al-2}u^2\df x
		&&=\int_{\R^{N-1}}\int_0^M|u(x',x_N)-u(x',0)|^2x_N^{\al-2}\df x_N\df x'\\
		&&=\int_{\R^{N-1}}\int_0^M\left(\int_0^{x_N}\f{\pt u}{\pt x_N}(x',s)\df s\right)^2x_N^{\al-2}\df x_N\df x' \\
		&&\leq \int_{\R^{N-1}}\int_0^M\left(\int_0^{x_N}s^\be \left|\f{\pt u}{\pt x_N}\right|^2(x',s)\df s\right)\left(\int_0^{x_N}s^{-\be}\df s\right)x_N^{\al-2}\df x_N\df x'\\
		&&=\f{1}{1-\be}\int_{\R^{N-1}}\int_0^M x_N^{\al-\be-1}\left(\int_0^{x_N}s^\be\left|\f{\pt u}{\pt x_N}\right|^2(x',s)\df s\right)\df x_N\df x'\\
		&&=\f{1}{1-\be}\int_{\R^{N-1}}\int_0^M\int_{s}^M s^\be \left|\f{\pt u}{\pt x_N}\right|^2(x',s)x_N^{\al-\be-1}\df x_N\df s\df x'\\
		&&\leq \f{1}{(1-\be)(\be-\al)}\int_{\R^{N-1}}\int_0^M s^\al \left|\f{\pt u}{\pt x_N}\right|^2\df s\df x'\\
		&&=\f{1}{(1-\be)(\be-\al)}\int_\Om x_N^\al \left|\f{\pt u}{\pt x_N}\right|^2\df x
\end{eqnarray*}
due to   $\al\in (0,1)$, and
\begin{equation*}
	\int_s^Mx_N^{\al-\be-1}\df x_N=\f{1}{\be-\al}\left(\f{1}{s^{\be-\al}}-\f{1}{M^{\be-\al}}\right)\leq \f{1}{\be-\al}s^{\al-\be}.
\end{equation*}
Hence,
\begin{equation*}
	\int_\Om x_N^{\al-2}u^2\df x\leq \f{4}{(1-\al)^2}\int_\Om x_N^\al |\pt_{x_N} u|^2\df x
\end{equation*}
by setting $\be=\f{1+\al}{2}$. This completes the proof of the lemma.
\end{proof}

\begin{corollary}\label{11.29.C1}
 The following two statements are valid:

	(i)~ For all $u\in H_0^1(\Om; w)$,
	\begin{equation*}
		\int_\Om u^2\df x\leq C\int_\Om  \nabla u\cdot A\nabla u\df x,
	\end{equation*}
	where the constants $C>0$ depends only on $\al$ and $M$. Hereafter, we will employ the norm
	\begin{equation*}
		\|u\|_{H_0^1(\Om;w)}=\left(\int_\Om \nabla u\cdot A\nabla u\df x\right)^\f{1}{2}
	\end{equation*}
	on $H_0^1(\Om;w)$;
	
	(ii) The embedding $H_0^1(\Om;w)\hra L^2(\Om)$ is compact.
\end{corollary}

\begin{proof}
	(i) follows directly from Lemma \ref{11.29.L1}.
	
	(ii) Let $\{u_n\}_{n\in\N}\s H_0^1(\Om;w)$ be a bounded sequence, i.e., $\|u_n\|_{H_0^1(\Om;w)}\leq C$ for all $n\in\N$
 with a fixed constant  $C>0$. Then, there exists a subsequence of $\{u_n\}_{n\in\N}$, still denoted by itself, and $u_0\in H_0^1(\Om;w)$ such that
	\begin{equation*}
		u_n\ra u_0 \mbox{ weakly in } H_0^1(\Om;w).
	\end{equation*}
	 Now, we will demonstrate $u_n\ra u_0$ in $L^2(\Om)$
by considering an abstract subsequence. We assume  $u_0=0$. If not, we replace $u_n$ with $u_n-u_0$ for all $n\in\N$.
	
	Let $\e>0$. One one side, for all $\de\in (0, \de_0)$, we obtain
	\begin{equation*}
		\begin{split}
		\int_{\Om-\Om_\de} u_n^2\df x
		&=\int_{\Om-\Om_\de} x_N^{2-\al}x_N^{\al-2}u_n^2\df x<\de^{2-\al}\|u_n\|_{H_0^1(\Om;w)}^2\leq C^2\de^{2-\al}.
		\end{split}
	\end{equation*}
	Selecting $\wh\de_0\in (0,\de_0)$  such that for all $\de\in (0,\wh \de_0]$ and all $n\in\N$ we have  $\int_{\Om-\Om_\de}u_n^2\df x<\f{1}{4}\e^2$. On the other hand, we consider the sequence $\{u_n|_{\Om_{\wh \de_0}}\}_{ n\in\N}$.   Hence, from
	\begin{equation*}
		\wh \de_0^\al \int_{\Om_{\wh\de_0}} |\nabla u|^2\df x\leq \int_{\Om_{\wh \de}} |\nabla u|^2x_N^\al \df x\leq C
	\end{equation*}
	we deduce that  $\{u_n|_{\Om_{\wh \de_0}}\}_{n\in\N}\s H^1(\Om_{\wh\de_0})$ is a bounded sequence.
By the classical Sobolev compact embedding theorem, there exists a subsequence $\{u_{n_k}|_{\Om_{\wh\de_0}}\}_{k\in\N}$ of  $\{u_n|_{\Om_{\wh\de_0}}\}_{n\in\N}$  such that $u_{n_k}|_{\Om_{\wh\de_0}}\ra 0$ strongly in $L^2(\Om_{\wh\de_0})$.  Moreover, there exists $k_0\in\N$ such that for all $k\geq k_0$, we have
	\begin{equation*}
		\left\|u_{n_k}|_{\Om_{\wh\de_0}}\right\|_{L^2(\Om_\e)}<\f{1}{2}\e.
	\end{equation*}
From these two aspects, we obtain the desired result.
\end{proof}

\subsection{Solution of degenerate elliptic equation, regularity}

 In this subsection, we demonstrate the existence of a weak solution for the degenerate elliptic equation, a result that is well-established in the literature.

Consider the equation given by
\begin{equation}\label{12.05.1}
	\begin{cases}
		-\mathcal{A} u = f, & \text{in } \Omega, \\
		u = 0, & \text{on } \partial\Omega,
	\end{cases}
\end{equation}
where $\mcA$ is given by \dref{02.05.1} and  $f \in L^2(\Omega)$. We define a weak solution $u \in H_0^1(\Omega;w)$ of \eqref{12.05.1} with respect to $f$ as one that satisfies
\begin{equation}
	\int_{\Omega} \nabla u \cdot A \nabla v \, dx = \int_{\Omega} fv \, dx
\end{equation}
for all $v \in H_0^1(\Omega;w)$.

\begin{lemma}\label{12.05.L1}
	Let $f \in L^2(\Omega)$. Then, there exists a unique solution $u \in H_0^1(\Omega;w) \cap H^2(\Omega;w)$ to the equation \eqref{12.05.1}.
\end{lemma}

\begin{proof}
	Define the bilinear form
	\begin{equation*}
		B[u,v] = \int_{\Omega} \nabla u \cdot A \nabla v \, dx, \quad \text{for } u, v \in H_0^1(\Omega;w).
	\end{equation*}
	It is clear that $|B[u,v]| \leq \|u\|_{H_0^1(\Omega;w)}\|v\|_{H_0^1(\Omega;w)}$, and
	\begin{equation*}
		B[u,u] = \int_{\Omega} \nabla u \cdot A \nabla u \, dx = \|u\|_{H_0^1(\Omega;w)}^2.
	\end{equation*}
	From Corollary \ref{11.29.C1} (i), we have
	\begin{equation*}
		\int_{\Omega} fv \, dx \leq \|f\|_{L^2(\Omega)}\|v\|_{L^2(\Omega)} \leq C\|f\|_{L^2(\Omega)}\|v\|_{H_0^1(\Omega;w)}.
	\end{equation*}
	This implies that the functional $f: H_0^1(\Omega;w) \to \mathbb{R}$ is bounded and linear. Consequently, by the Lax-Milgram theorem, there exists a unique weak solution to \eqref{12.05.1}.
	
	Finally, since $f \in L^2(\Omega)$, it follows that $u \in H^2(\Omega;w)$. This completes the proof of the lemma.
\end{proof}

\begin{remark}\label{12.27.R1}
	It can be readily verified that $\mathcal{A}: H_0^1(\Omega;w) \to H^{-1}(\Omega;w)$ define by \dref{02.05.1} is a bounded linear operator. Indeed, for each $u \in H_0^1(\Omega;w)$ and every $v \in C_0^\infty(\Omega)$, we have $\langle \mathcal{A} u, v \rangle_{H^{-1}(\Omega;w), H_0^1(\Omega;w)} = (u, v)_{H_0^1(\Omega;w)}$ in the sense of distributions. Therefore, for all $u, v \in H_0^1(\Omega;w)$, we have
	\begin{equation*}
		\langle \mathcal{A} u, v \rangle_{H^{-1}(\Omega;w), H_0^1(\Omega;w)} = (u, v)_{H_0^1(\Omega;w)}
	\end{equation*}
	by the density of $C_0^\infty(\Omega)$ in $H_0^1(\Omega;w)$. Consequently,
	\begin{equation*}
		\begin{split}
			\|\mathcal{A} u\|_{H^{-1}(\Omega;w)} &= \sup_{\|v\|_{H_0^1(\Omega;w)} \leq 1} \langle \mathcal{A} u, v \rangle_{H^{-1}(\Omega;w), H_0^1(\Omega;w)} \\
			&\leq \sup_{\|v\|_{H_0^1(\Omega;w)} \leq 1} (u, v)_{H_0^1(\Omega;w)} \leq \|u\|_{H_0^1(\Omega;w)}.
		\end{split}
	\end{equation*}
\end{remark}

\subsection{Eigenvalues and eigenfunctions}

To derive estimates for the weak solution of \eqref{01.08.1}, we first introduce the eigenvalues and eigenfunctions associated with the corresponding differential operator $\mcA$ defined in {\color{red} \eqref{02.05.1}}.

\begin{notation}\label{11.29.N1}
	 Based on Corollary \ref{11.29.C1}, the partial differential operator
$\mcA$ defined by  \dref{02.05.1}
 possesses a discrete point spectrum given by
	\begin{equation*}
		0<\la_1<\la_2\leq \la_3\leq \cdots \ra +\iy.
	\end{equation*}
	 Furthermore, from Corollary \ref{11.29.C1} (i), we obtain
	\begin{equation*}
		0<\la_1=\inf_{0\neq u\in H_0^1(\Om;w)}\f{\int_\Om \nabla u\cdot A\nabla u\df x}{\int_\Om u^2\df x}.
	\end{equation*}
	 We designate  $\Phi_n$ (for $n\in\N)$  as the eigenfunction of
$\mcA$
 corresponding to the eigenvalue  $\la_n$. That is, $\la_n$ and $\Phi_n$ satisfies the following equation
	\begin{equation}\label{12.23.3}
		\begin{cases}
		\mcA \Phi_n=\la_n \Phi_n, &\mbox{in }\Om,\\
		\Phi_n=0, &\mbox{on }\pt\Om.
		\end{cases}
	\end{equation}
	 Hereafter, we denote $\{\Phi_n\}_{n\in\N}$  as the orthonormal basis of  $L^2(\Om)$.  Additionally,
 $\{\Phi_n\}_{n\in\N}$  forms an orthogonal subset of $H_0^1(\Om;w)$ (see, e.g.,
  \cite[Theorem 7, Appendix D, p. 728]{Evans} or Lemma \ref{06.28.L1} below). It is important to note that
$\Phi_n\in D(\mcA)$  for all $n\in\N$,  as established by Lemma \ref{12.05.L1}.
\end{notation}

\begin{lemma}\label{06.28.L1}  Let $\mcA$ be defined as \dref{02.05.1}.
 	Let $u=\sum_{i=1}^\iy u_i \Phi_i\in H_0^1(\Om;w)$ with $u_i=(u,\Phi_i)_{L^2(\Om)}$ for all $i\in\N$.
Then,  $\nabla u=\sum_{i=1}^\iy u_i\nabla \Phi_i$ and $\|u\|_{H_0^1(\Om;w)}=(\sum_{i=1}^\iy u_i^2\la_i)^\f{1}{2}$. Moreover,
	\begin{equation*}
		u\in H^2(\Om;w)\Lra \sum_{i=1}^\iy u_i^2\la_i^2<\iy,
	\end{equation*}
	and
	\begin{equation*}
		\mcA u=\sum_{i=1}^\iy u_i\la_i\Phi_i, \mbox{ and } \|\mcA u\|_{L^2(\Om)}=\left(\sum_{i=1}^\iy u_i^2\la_i^2\right)^\f{1}{2}.
	\end{equation*}
\end{lemma}

\begin{proof}
	From \eqref{12.23.3}, we obtain
	\begin{equation}\label{06.04.3}
		\int_\Om \nabla \Phi_k\cdot A\nabla \Phi_l\df x=\de_{kl}\la_k \mbox{ for } k,l\in \N.
	\end{equation}
	 We now prove that $\{\la_k^{-\f{1}{2}}\Phi_k\}_{k=1}^\iy$ forms an orthonormal basis of $H_0^1(\Om;w)$.
	
From \eqref{06.04.3},  it is easy to verify that $\{\la_k^{-\f{1}{2}}\Phi_k\}_{k=1}^\iy$ is an orthonormal subset of $H_0^1(\Om;w)$.
 To show that it is a basis, assume, for the sake of contradiction, that there exists  $0\neq u\in H_0^1(\Om;w)$ such that
	\begin{equation*}
		\int_\Om \nabla \Phi_k\cdot A\nabla u\df x=0 \mbox{ for all }k\in\N.
	\end{equation*}
	
Since $\{\Phi_k\}_{k\in\N}$ is an orthonormal basis of $L^2(\Om)$, for $u\in H_0^1(\Om;w)$,  we can write
	\begin{equation}\label{06.04.4}
		u=\sum_{k=1}^\iy d_k\Phi_k \mbox{ where }  d_k=(u, \Phi_k)_{L^2(\Om)}, k\in\N.
	\end{equation}
	Then  we have
	\begin{equation*}
		0=(u,\Phi_k)_{H_0^1(\Om;w)}=\int_\Om \nabla \Phi_k\cdot A\nabla u\df x=\la_k\int_\Om \Phi_ku\df x =\la_kd_k,
	\end{equation*}
	which implies $d_k=0$. This leads to $u=0$, a contradiction.
	
	Since $u\in H_0^1(\Om;w)$,  we have
	\begin{equation*}
		\nabla u=\sum_{i=1}^\iy e_i\la_i^{-\f{1}{2}}\nabla\Phi_k, \mbox{ and } \|u\|_{H_0^1(\Om;w)}^2=\sum_{i=1}^\iy e_i^2,
	\end{equation*}
	and
	\begin{equation*}
		e_i=\int_\Om \nabla u\cdot \left(A \la_i^{-\f{1}{2}}\nabla\Phi_i\right) \df x=\la_i^{-\f{1}{2}}\int_\Om \nabla u\cdot A\nabla \Phi_i\df x=\la_i^\f{1}{2}\int_\Om u\Phi_i\df x=\la_i^\f{1}{2}u_i,
	\end{equation*}
	hence,  $\nabla u=\sum_{i=1}^\iy u_i\nabla\Phi_i$. Moreover, $\|u\|_{H_0^1(\Om;w)}=(\sum_{i=1}^\iy \la_iu_i^2)^\f{1}{2}$.
	
	Let $u\in H^2(\Om;w)$. Take $\vp_n=\sum_{i=1}^n u_i\Phi_i\in H_0^1(\Om;w)\cap H^2(\Om;w)$ for each $n\in\N$. Then  from
	\begin{equation*}
		\begin{split}
			(\mcA u, \mcA\vp_n)_{L^2(\Om)}
			&=\sum_{i=1}^n u_i\la_i \int_\Om (\mcA u)\Phi_i\df x=\sum_{i=1}^n u_i\la_i\int_\Om u\mcA\Phi_i\df x\\
			&=\sum_{i=1}^n u_i\la_i^2\int_\Om u\Phi_i\df x=\sum_{i=1}^n u_i^2\la_i^2
		\end{split}
	\end{equation*}
	and $\|\mcA\vp_n\|_{L^2(\Om)}^2=\sum_{i=1}^n u_i^2\la_i^2$, by the Cauchy-Schwarz inequality, we get
	\begin{equation*}
		\sum_{i=1}^nu_i^2\la_i^2\leq \|\mcA u\|_{L^2(\Om)}^2
	\end{equation*}
	for all $n\in\N$. This implies that $\sum_{i=1}^\iy u_i^2\la_i^2\leq \|\mcA u\|_{L^2(\Om)}^2<\iy$.

	Let $\sum_{i=1}^\iy u_i^2\la_i^2<\iy$. For each $\vp\in C_0^\iy(\Om)$,  we have
	\begin{equation*}
		\begin{split}
			(\mcA u, \vp)_{L^2(\Om)}
			&=\int_\Om \nabla u\cdot A\nabla\vp\df x=\sum_{i=1}^\iy u_i \int_\Om \nabla\Phi_i\cdot A\nabla\vp\df x=\sum_{i=1}^\iy u_i\la_i\int_\Om \Phi_i\vp\df x.
		\end{split}
	\end{equation*}
	Note that
	\begin{equation*}
		\int_\Om \left(\sum_{i=1}^n u_i\la_i\Phi_i\right)^2\df x=\sum_{i=1}^n u_i^2\la_i^2\leq \sum_{i=1}^\iy u_i^2\la_i^2<\iy \mbox{ for all } n\in\N,
	\end{equation*}
	i.e., $\sum_{i=1}^\iy u_i\la_i\Phi_i\in L^2(\Om)$. Hence
	\begin{equation*}
		(\mcA u, \vp)_{L^2(\Om)}=\left(\sum_{i=1}^\iy u_i\la_i\Phi_i, \vp\right)_{L^2(\Om)}.
	\end{equation*}
	This implies that $\|\mcA u\|_{L^2(\Om)}\leq \sum_{i=1}^\iy u_i^2\la_i^2$ and $\mcA u=\sum_{i=1}^\iy u_i\la_i\Phi_i$.
\end{proof}

\subsection{Existence of weak solution}

 In this subsection, we establish the existence of a weak solution to equation \eqref{01.08.1} and elucidate the dependency of the constants in the corresponding a priori estimates.

\begin{definition}\label{01.27.D1}
	We define a weak solution $y$   of the equation \eqref{01.08.1} with respect to  $(y^0,f)$
	\begin{equation*}
		y\in L^2(0,T; H_0^1(\Om;w)), \mbox{ with } \pt_ty\in L^2(0,T; H^{-1}(\Om;w))
	\end{equation*}
	  satisfying
\begin{itemize}

\item[(i)]  for every $v\in H_0^1(\Om;w)$, and for almost every $t\in [0,T]$
	\begin{equation*}
		\lg \pt_ty, v\rg_{H^{-1}(\Om;w), H_0^1(\Om;w)}+(y,v)_{H_0^1(\Om;w)}=(f,v)_{L^2(\Om)};
	\end{equation*}
	
\item[(ii)] $y(0)=y^0$.
\end{itemize}
\end{definition}

 The following theorem represents one of the primary results of this paper.

\begin{theorem}\label{01.11.T1}
	Let $\mcA$ be given by \dref{02.05.1}. Let $\al\in (0,1)$, $y^0\in L^2(\Om)$ and $f\in L^2(Q)$. Then  the equation \eqref{01.08.1}
 with respect to $(y^0,f)$ has a unique weak solution
	\begin{equation*}
		y\in L^2(0,T; H_0^1(\Om;w)), \mbox{ with } \pt_ty\in L^2(0,T; H^{-1}(\Om;w))
	\end{equation*}
	satisfies the estimate
	\begin{equation}\label{01.11.7}
		\begin{split}
			&\esssup_{t\in [0,T]}\|y(t)\|_{L^2(\Om)}+\|y\|_{L^2(0,T;H_0^1(\Om;w))}+\|\pt_ty\|_{L^2(0,T; H^{-1}(\Om;w))}\\
			&\leq C\left(\|f\|_{L^2(Q)}+\|y^0\|_{L^2(\Om)}\right),
		\end{split}
	\end{equation}
	where the constant $C>0$ depends only on $\al, N, T$ and $M$ and the shape of  $\Om_{\de_0}$.
	
	 Furthermore, if $y^0\in H_0^1(\Om;w)$,  then the weak solution
	\begin{equation*}
		y\in L^2(0,T; D(\mcA))\cap L^2(0,T; H_0^1(\Om;w)), \mbox{ with } \pt_ty\in L^2(Q),
	\end{equation*}
	 satisfies the estimate
	\begin{equation}\label{01.11.8}
		\begin{split}
			&\esssup_{t\in [0,T]}\|y(t)\|_{H_0^1(\Om;w)}+\|y\|_{L^2(0,T; H^2(\mcO(\Ga^+,\de_0)))}+\|y\|_{L^2(0,T; D(\mcA))}+\|\pt_ty\|_{L^2(Q)}\\
			&\leq C\left(\|f\|_{L^2(Q)}+\|y^0\|_{H_0^1(\Om;w)}\right),
		\end{split}
	\end{equation}
	where the constant $C>0$ depends only on $\al, N, T$ and $M$,  and the shape of  $\Om_{\de_0}$.
\end{theorem}

\begin{proof}
	 We prove this theorem through the following six steps.
	
	{\it Step 1: Galerkin method.}
	
	Let $m\in\N^*$, and define
	\begin{equation*}
		y^m(x,t)=\sum_{n=1}^m y_n^m(t)\Phi_n(x),
	\end{equation*}
	where $y_n^m(t), t\in [0,T], n=1,\cdots, m$, is the solution to the equation
	\begin{equation}\label{01.11.1}
		\f{\df}{\df t}y_n^m(t)+\la_n y_n^m(t)=f_n(t), \mbox{ with } f_n(t)=(f,\Phi_n)_{L^2(\Om)}, n\in\N,
	\end{equation}
	with initial data
	\begin{equation}\label{01.11.2}
		y_n^m(0)=(y^0,\Phi_n)_{L^2(\Om)}, \ n=1,\cdots, m.
	\end{equation}
	 It is evident that
	\begin{equation*}
		y_n^m(t)=y_n^m(0)e^{-\la_nt}+\int_0^te^{\la_n(s-t)}f_n(s)\df s, \ n=1,\cdots, m.
	\end{equation*}
	Note that \eqref{01.11.1} is equivalent to the following equality
	\begin{equation}\label{01.11.3}
		(\pt_ty^m,\Phi_n)_{L^2(\Om)}+(y^m,\Phi_n)_{H_0^1(\Om;w)}=(f,\Phi_n)_{L^2(\Om)}, \ t\in [0,T], n=1,\cdots, m.
	\end{equation}
	
	{\it Step 2: Energy estimate.}
	
	Multiplying \eqref{01.11.3} by $y_n^m(t)$ and summing for $n=1,\cdots, m$,  we obtain
	\begin{equation*}
		(\pt_ty^m, y^m)_{L^2(\Om)}+(y^m,y^m)_{H_0^1(\Om;w)}=(f,y^m)_{L^2(\Om)} \mbox{ for a.e.}\ t\in [0,T].
	\end{equation*}
	Noting  that $(\pt_ty^m,y^m)_{L^2(\Om)}=\f{1}{2}\f{\df}{\df t}\|y^m\|_{L^2(\Om)}^2$,  and  $(f,y^m)_{L^2(\Om)}\leq C\|f\|_{L^2(\Om)}^2+C\|y^m\|_{L^2(\Om)}^2\leq C\|f\|_{L^2(\Om)}^2+\f{1}{2}\|y^m\|_{H_0^1(\Om;w)}^2$ by   Corollary \ref{11.29.C1} (i), we get
	\begin{equation*}
		\f{\df}{\df t}\|y^m\|_{L^2(\Om)}^2+\|y^m\|_{H_0^1(\Om;w)}^2\leq C\|f\|_{L^2(\Om)}^2,
	\end{equation*}
	where the constant $C>0$ depends only on $\al, N$ and $M$. Integrating over $[0,t]$ for any $t\in (0,T]$, we obtain
	\begin{equation}\label{01.11.4}
		\begin{split}
			\esssup_{t\in [0,T]}\|y^m(t)\|_{L^2(\Om)}+\|y^m\|_{L^2(0,T; H_0^1(\Om;w))}\leq \|y^0\|_{L^2(\Om)}+C\|f\|_{L^2(Q)},
		\end{split}
	\end{equation}
	where the constant $C>0$ depends only on $\al, N$ and $M$.
	
	For each $v\in H_0^1(\Om;w)$ with $\|v\|_{H_0^1(\Om;w)}\leq 1$, write $v=v^1+v^2$, where $v^1\in \Span \{\Phi_n\}_{n=1}^m$ and $(v^2, \Phi_n)=0$ for $k=1,\cdots, m$. Then $\|v^1\|_{H_0^1(\Om;w)}\leq \|v\|_{H_0^1(\Om;w)}\leq 1$, and from \eqref{01.11.3} we obtain
	\begin{equation*}
		\lg \pt_ty^m, v\rg_{H^{-1}(\Om;w),H_0^1(\Om;w)}=(\pt_ty^m,v)_{L^2(\Om)}=(\pt_ty^m, v^1)_{L^2(\Om)}=(f,v^1)_{L^2(\Om)}-(y^m, v^1)_{H_0^1(\Om;w)}.
	\end{equation*}
	This implies that
	\begin{equation*}
		\|\pt_ty^m\|_{H^{-1}(\Om;w)}\leq C\|f\|_{L^2(\Om)}+\|y^m\|_{H_0^1(\Om;w)}
	\end{equation*}
	by Corollary \ref{11.29.C1} (i). Hence,
	\begin{equation}\label{01.11.5}
		\|\pt_ty^m\|_{L^2(0,T; H^{-1}(\Om;w))}\leq \|y^0\|_{L^2(\Om)}+ C\|f\|_{L^2(Q)},
	\end{equation}
	where the constant $C>0$ depends only on $\al, T, N$ and $M$.
	
	{\it Step 3: Approximation.}
	
	 From \eqref{01.11.4} and \eqref{01.11.5}, there exists a function  $z\in L^2(0,T; H_0^1(\Om;w))$ with $\pt_tz\in L^2(0,T; H^{-1}(\Om;w))$ such that
	\begin{equation}\label{01.11.6}
		\begin{split}
			y^m
			&\ra z\mbox{ weak}^* \hbox{ in } L^\iy(0,T; L^2(\Om)),\\
			y^m
			&\ra z\mbox{ weakly in }L^2(0,T; H_0^1(\Om;w)), \\
			\pt_ty^m
			&\ra \pt_t z\mbox{ weakly in }L^2(0,T;H^{-1}(\Om;w)).
		\end{split}
	\end{equation}
	Moreover, we have \eqref{01.11.7} (replace $y$ by $z$), and $z\in C([0,T]; L^2(\Om))$.
	
	{\it Step 4:  $z$ is a solution of \eqref{01.08.1}.}
	
	Fix $k\in\N$. Choose a function
	\begin{equation*}
		\psi(t)=\sum_{n=1}^k \psi_n(t)\Phi_n, \mbox{ with } \psi_n(t) \mbox{ are smooth functions on } [0,T],\ n=1,\cdots, k.
	\end{equation*}
	Then for all $m\geq k$, multiplying \eqref{01.11.3} by $\psi_n(t)$, summing for  $n=1,\cdots, k$, and integrating
over $t\in [0,T]$, we get
	\begin{equation}\label{01.11.11}
		\begin{split}
			\int_0^T \lg \pt_t y^m, \psi\rg_{H^{-1}(\Om;w),H_0^1(\Om;w)}\df t+\iint_Q w\nabla y^m \cdot\nabla \psi\df x\df t=\iint_Q f\psi\df x\df t.
		\end{split}
	\end{equation}
	 By taking an abstract subsequence and letting
 $m\ra\iy$, we obtain
	\begin{equation}\label{01.11.9}
		\int_0^T\lg \pt_t z,\psi\rg_{H^{-1}(\Om;w),H_0^1(\Om;w)}\df t+\iint_Q w\nabla z\cdot\nabla\psi\df x\df t=\iint_Q f\psi\df x\df t
	\end{equation}
	by \eqref{01.11.6}.  Noting that the set of functions $\psi$ is dense in $L^2(0,T; H_0^1(\Om;w))$, hence \eqref{01.11.9} holds for all $\psi\in L^2(0,T; H_0^1(\Om;w))$, and
	\begin{equation*}
		\lg \pt_tz, v\rg_{H^{-1}(\Om;w),H_0^1(\Om;w)}+(z, v)_{H_0^1(\Om;w)}=(f,v)_{L^2(\Om)}
	\end{equation*}
	for all $ v\in H_0^1(\Om;w)$ and almost every  $t\in [0,T]$.
	
	Now, we prove $z(0)=y^0$. Indeed, from \eqref{01.11.9} and $z\in C([0,T]; L^2(\Om))$,  we get
	\begin{equation}\label{01.11.10}
		-\iint_Q z\pt_t\psi\df x\df t+\iint_{Q}w\nabla z\cdot \nabla\psi\df x\df t=\iint_Q f\psi\df x\df t+(z(0),\psi(0))_{L^2(\Om)}
	\end{equation}
	for each $\psi\in C^1([0,T]; H_0^1(\Om;w))$ with $\psi(T)=0$. Similarly, from \eqref{01.11.11}, we obtain
	\begin{equation*}
		\begin{split}
			-\iint_Q y^m\pt_t\psi\df x\df t+\iint_Qw\nabla y^m\cdot \nabla\psi\df x\df t=\iint_Qf\psi\df x\df t+(y^m(0),\psi(0))_{L^2(\Om)}.
		\end{split}
	\end{equation*}
	 By taking an abstract subsequence and letting  $m\ra \iy$, from \eqref{01.11.6} and  \eqref{01.11.2}, we get
	\begin{equation*}
		-\iint_Q z\pt_t\psi\df x\df t+\iint_Q w\nabla z\cdot\nabla \psi\df x\df t=\iint_Q f\psi\df x\df t+(y^0,\psi(0))_{L^2(\Om)}.
	\end{equation*}
	 This, combined with \eqref{01.11.10} and the arbitrariness of  $\psi$, yields  $z(0)=y^0$.
	
	{\it Step 5:  Uniqueness.}
	
	 It suffices to show that the weak solution of \eqref{01.08.1} corresponding to $(y^0 = 0,f = 0)$ is identically zero.	
	Denote $\xi=y-z$. Then $\xi$ is the solution of \eqref{01.08.1} with respect to $(0,0)$, and hence from Definition \ref{01.27.D1} (ii) (replace $v$ by $\xi$) we get
	\begin{equation*}
		\f{1}{2}\f{\df}{\df t}\|\xi\|_{L^2(\Om)}^2+\|\xi\|_{H_0^1(\Om;w)}^2=0.
	\end{equation*}
	This implies $\xi=0$ by integrating over $(0,t)$ for each $t\in [0,T]$.
	
	{\it Step 6: Improved regularity.}
	
	Multiplying \eqref{01.11.3} by $\f{\df}{\df t}y_n^m$, summing for $n=1,\cdots, m$, we get
	\begin{equation*}
		(\pt_ty^m, \pt_ty^m)_{L^2(\Om)}+(y^m, \pt_ty^m)_{H_0^1(\Om;w)}=(f,\pt_ty^m)_{L^2(\Om)} \mbox{ for almost every }\ t\in [0,T].
	\end{equation*}
	Noting  that $(y^m, \pt_ty^m)_{H_0^1(\Om;w)}=\f{1}{2}\f{\df }{\df t}\|y^m\|_{H_0^1(\Om;w)}^2$, we obtain
	\begin{equation*}
		\begin{split}
			\|\pt_ty^m\|_{L^2(\Om)}^2+\f{\df}{\df t}\|y^m\|_{H_0^1(\Om;w)}^2\leq \|f\|_{L^2(\Om)}^2 \mbox{ for a.e.}\ t\in [0,T].
		\end{split}
	\end{equation*}
	Integrating over  $(0,t)$ for each $t\in [0,T]$, we get
	\begin{equation}\label{01.27.5}
		\begin{split}
			&\esssup_{t\in [0,T]}\|y^m(t)\|_{H_0^1(\Om;w)}^2+\|\pt_ty^m\|_{L^2(Q)}^2\\
			&\leq 2\|f\|_{L^2(Q)}^2+\|y^m(0)\|_{H_0^1(\Om;w)}^2\\
			&=2\|f\|_{L^2(Q)}^2+2\sum_{n=1}^m (y_n^0)^2\la_n\leq 2\|f\|_{L^2(Q)}^2+2\|y^0\|_{H_0^1(\Om;w)}^2.
		\end{split}
	\end{equation}
	By Lemma \ref{06.28.L1}. This implies that
	\begin{equation}\label{01.11.13}
		\begin{split}
			\esssup_{t\in [0,T]}\|y(t)\|_{H_0^1(\Om;w)}^2+\|\pt_ty\|_{L^2(Q)}^2\leq 2\|f\|_{L^2(Q)}^2+2\|y^0\|_{H_0^1(\Om;w)}^2
		\end{split}
	\end{equation}
	by \eqref{01.11.6} and {\it Step 4}.
	
	Now, choose $\zeta=\zeta(x_N)\in C_0^\iy(\R^N), 0\leq \zeta\leq 1$ such that
	\begin{equation*}
		\zeta =0 \mbox{ on } \left\{x_N<\de_0\right\},\quad \zeta=1 \mbox{ on } \{x_N\geq 2\de_0\},\quad |\zeta'|\leq C\de_0^{-1}, \mbox{ and } |\zeta''|\leq C\de_0^{-2},
	\end{equation*}
	where the constant $C>0$ is  absolute. Then $z=\zeta y^m$ is the solution of the following elliptic equation
	\begin{equation*}
		\begin{cases}
			\mcA z=\zeta f^m-\zeta\pt_{t}y^m+y^m\mcA \zeta-2A\nabla \zeta\cdot \nabla y^m, &\mbox{on } \Om_{\de_0},\\
			z=0, &\mbox{on } \pt \Om_{\de_0},
		\end{cases}
	\end{equation*}
where
	$f^m=\sum_{n=1}^m(f,\Phi_n)_{L^2(\Om)}\Phi_n$.
It is evident that this equation is a uniformly elliptic equation.
By the standard regularity for uniformly elliptic equations
\cite[Theorem 1, p. 327; Theorem 4, p. 334]{Evans}, from \eqref{01.27.5},
we get, for almost every $t\in [0,T]$, that
	\begin{equation*}
		\begin{split}
			\|z\|_{H^2(\Om_{\de_0})}\leq C\left(\|f(t)\|_{L^2(\Om)}+\|\pt_ty^m(t)\|_{L^2(\Om)}+\|y^m(t)\|_{H_0^1(\Om;w)}\right),
		\end{split}
	\end{equation*}
	where the constant $C>0$ depending only on $\al$ and the shape of  $\Om_{\de_0}$. Hence, from \eqref{01.11.13}, we get
	\begin{equation*}
		\|y^m\|_{L^2(0,T;H^2(\mcO(\Ga^+,\de_0)))}\leq C\left(\|f\|_{L^2(Q)}+\|y^0\|_{H_0^1(\Om;w)}\right).
	\end{equation*}
	From \eqref{01.11.6} we get
	\begin{equation}\label{01.27.6}
		\|y\|_{L^2(0,T; H^2(\mcO(\Ga^+,\de_0))}\leq C\left(\|f\|_{L^2(Q)}+\|y^0\|_{H_0^1(\Om;w)}\right),
	\end{equation}
	where the constant $C>0$ depends only on $\al$, $T, M$ and the shape of  $\Om_{\de_0}$.
  Finally, from \eqref{01.11.13}, we have
	\begin{equation*}
		\|\mcA y\|_{L^2(Q)}\leq \|f\|_{L^2(Q)}+\|\pt_ty\|_{L^2(Q)}\leq 3\|f\|_{L^2(Q)}^2+2\|y^0\|_{H_0^1(\Om;w)}^2.
	\end{equation*}
	This together with \eqref{01.11.13} yields  \eqref{01.11.8}, where the constant $C>0$ depends only on $\al, N, T$ and $M$ and the shape of  $\Om_{\de_0}$. We thus complete the proof of the theorem.
\end{proof}

\begin{corollary}\label{01.12.C1}
	Let $y^0\in H_0^1(\Om;w)$ satisfy $\mcA y^0\in H_0^1(\Om;w)$. Let $y$ be the weak solution of \eqref{01.08.1} with respect to $(y^0,0)$. Then,
	\begin{equation*}
		y\in L^2(0, T; D(\mcA^2)), \mbox{ with } \pt_t y\in L^2(0, T; D(\mcA)),
	\end{equation*}
	and
	\begin{equation*}
		\begin{split}
			\|\pt_ty\|_{L^2(0, T; L^2(\mcO(\Ga^+,\de_0)))}\leq C\|\mcA y^0\|_{H_0^1(\Om;w)},
		\end{split}
	\end{equation*}
	where the constant $C>0$ depends only on $\al, T, N$ and $M$ and the shape of $\Om_{\de_0}$.
\end{corollary}

\begin{proof}
We note that $y \in D(\mathcal{A}^2)$ if $y \in D(\mathcal{A})$ and $\mathcal{A}y \in D(\mathcal{A})$.

	From Lemma \ref{06.28.L1} we have $\mcA y^0=\sum_{n=1}^\iy (y^0, \Phi_n)_{L^2(\Om)}\la_n\Phi_n$. Noting that
	\begin{equation*}
		\wh y=\sum_{n=1}^\iy (y^0,\Phi_n)_{L^2(\Om)}\la_ne^{-\la_nt}=\pt_t y,
	\end{equation*}
	is the weak solution of \eqref{01.08.1} with respect to $(\mcA y^0,0)$, from \eqref{01.11.8} and $\mcA \pt_ty=\pt_t\mcA y$, we obtain the corollary.
\end{proof}

\section{Shape design}\label{S3}

 In this section, we demonstrate that a degenerate parabolic equation can be approximated by a sequence of uniformly parabolic equations through a shape design approach.

Let $Q_\de=\Om_\de\ts (0,T)$, and  $w_\de=w|_{\Om_\de}, A_\de=A|_{\Om_\de}$ and  $\mcA_\de=\mcA|_{\Om_\de}$. We consider the following equation:
\begin{equation}\label{01.11.14}
	\begin{cases}
		\pt_t y_\de-\mcA_\de y_\de=f_\de, &\mbox{in }Q_\de,\\
		y_\de=0, &\mbox{on }\pt Q_\de, \\
		y_\de(0)=y_\de^0, &\mbox{in }\Om_\de,
	\end{cases}
\end{equation}
where $y_\de^0\in L^2(\Om_\de)$ and $f_\de\in L^2(\Om_\de)$.

Similar to Section \ref{S2}, we can derive the following solution spaces and results. We define
\begin{equation*}
	H^{1}(\Omega_\de;w_\de) = \left\{u \in L^2(\Omega_\de) \colon \int_{\Om_\de} \nabla u\cdot A_\de\nabla u\df x<+\iy\right\}.
\end{equation*}
The inner product on $H^1(\Omega_\de; w_\de)$ is given by
\begin{equation*}
	(u, v)_{H^1(\Omega_\de; w_\de)} =\int_{\Om_\de} uv\df x+\int_{\Om_\de} \nabla u\cdot A_\de\nabla v \df x,
\end{equation*}
and the norm is
\begin{equation*}
	\|u\|_{H^1(\Omega_\de; w_\de)} = \left(\int_{\Om_\de} u^2\, \mathrm{d}x +   \int_{\Om_\de} \nabla u\cdot A_\de\nabla u \mathrm{d}x\right)^{\frac{1}{2}}.
\end{equation*}
We define
\begin{equation*}
	H_{0}^1(\Omega_\de; w_\de) = \text{the closure of } \mathcal{D}(\Omega_\de) \text{ in } H^1(\Omega_\de; w_\de)
\end{equation*}
 We denote by $H^{-1}(\Omega_\de; w_\de)$ the dual space of  $H_{0}^1(\Omega_\de; w_\de)$
 with pivot space $L^2(\Om_\de)$, which is a subspace of $\mathcal{D}'(\Omega_\de)$.  It is well-known that
 $(H^1(\Omega_\de; w_\de), (\cdot, \cdot)_{H^1(\Omega_\de; w_\de)})$  is a Hilbert space and
 $(H^1(\Omega_\de; w_\de), \|\cdot\|_{H^1(\Omega_\de; w_\de)})$
 is a Banach space (\cite{GC,Heinonen}).
Now, we define
\begin{equation*}
	H^2(\Om_\de;w_\de)=\left\{u\in H^1(\Om_\de;w_\de)\colon \int_{\Om_\de} (\mcA_\de u)^2\df x<+\iy\right\}.
\end{equation*}
Its inner product   is defined by
\begin{equation*}
	(u,v)_{H^2(\Om;w)}=(u,v)_{H^1(\Om;w)}+\int_\Om (\mcA_\de u)(\mcA_\de v)\df x,
\end{equation*}
and its  norm is defined as
\begin{equation*}
	\|u\|_{H^2(\Om_\de)}=\left(\int_{\Om_\de} u^2 \df x+\int_{\Om_\de} \nabla u\cdot A_\de\nabla u\df x+\int_{\Om_\de} (\mcA_\de u)^2\df x\right)^\f{1}{2}.
\end{equation*}
It is also well-known that $(H^2(\Omega_\de; w_\de), (\cdot, \cdot)_{H^2(\Omega_\de; w_\de)})$ is a Hilbert space and $(H^2(\Omega_\de; w_\de), \|\cdot\|_{H^2(\Omega_\de; w_\de)})$ is a Banach space (\cite{GC,Heinonen}).

We denote
\begin{equation*}
	D(\mcA_\de)=H^2(\Om_\de;w_\de)\cap H_0^1(\Om_\de;w_\de).
\end{equation*}

We note that $H_0^1(\Omega_\delta; w_\delta) = H_0^1(\Omega_\delta)$, $H^1(\Omega_\delta; w_\delta) = H^1(\Omega_\delta)$, and $H^2(\Omega_\delta; w_\delta) = H^2(\Omega_\delta)$,   as a consequence of the standard improved regularity for uniformly elliptic equations. However, in general, $H^2(\Omega; w) \neq H^2(\Omega)$.

 Lemmas \ref{01.11.L1} and \ref{01.11.L2}, as well as Remark \ref{01.11.R1}, are equivalent to Lemma \ref{11.29.L1} and Corollary \ref{11.29.C1}, respectively.

\begin{lemma}\label{01.11.L1}
	Let $N \geq 2$ and $\alpha \in (0, 1)$. Then, for all $u \in H_0^1(\Omega_\de; w_\de)$, the following inequality holds:
	\begin{equation*}
		\int_{\Om_\de}x_N^{\al-2}u^2\df x\leq C\int_{\Om_\de} x_N^\al (\pt_{x_N}u)^2\df x,
	\end{equation*}
	where the constant $C>0$ depends only on $\al$.
\end{lemma}

\begin{remark}\label{01.11.R1}
	From Lemma \ref{01.11.R1},  we have
	\begin{equation}\label{01.11.12}
		\int_{\Om_\de} u^2\df x\leq C\int_{\Om_\de} \nabla u\cdot A\nabla u\df x,
	\end{equation}
	for all $u\in H_0^1(\Om_\de;w_\de)$, where the constant $C>0$ depends only on $\al$ and $M$. The second inequality is  the Poincar\'{e} inequality.
	In particular, the norm
	\begin{equation}\label{01.11.15}
		\|u\|_{H_0^1(\Omega_\de; w_\de)} = \left(\int_{\Omega_\de} \nabla u \cdot A_\de\nabla u  \mathrm{d}x\right)^{\frac{1}{2}}
	\end{equation}
	is an equivalent norm in $H_0^1(\Omega_\de; w_\de)$.  Hereafter, we use \eqref{01.11.15} to define the norm of $H_0^1(\Omega_\de; w_\de)$.
\end{remark}

\begin{lemma}\label{01.11.L2}
	The embedding $H_0^1(\Omega_\de; w_\de) \hookrightarrow L^2(\Omega_\de)$ is compact.
\end{lemma}

 We are now ready to define a weak solution to equation \eqref{01.11.14}.

\begin{definition}\label{01.11.D1}
	We call
	\begin{equation*}
		y_\de\in L^2(0,T; H_0^1(\Om_\de;w_\de)), \mbox{ with } \pt_ty_\de\in L^2(0,T; H^{-1}(\Om_\de;w_\de))
	\end{equation*}
	is a weak solution of \eqref{01.11.14}  with respect to $(y_\de^0,f_\de)$
	provided
	
(i) for every $v\in H_0^1(\Om_\de;w_\de)$ and almost every $t\in (0,T)$,  it has
\begin{equation*}
	\lg\pt_{t}y_\de, v\rg_{H^{-1}(\Om_\de;w_\de),H_0^1(\Om_\de;w_\de)}+\int_{\Om_\de} w_\de\nabla y_\de\cdot \nabla v\df x=\int_{\Om_\de} f_\de v\df x,
\end{equation*}

(ii) $y_\de(0)=y_\de^0$.
\end{definition}

The following theorem can be proved using the same argument as Theorem \ref{01.11.T1}.
We note that the constants in the estimates do not depend on the parameter $\delta \in (0, \delta_0)$.
 In fact, from the proof of Theorem \ref{01.11.T1}, we see that the constants appearing in the estimates depend only on the constants in Corollary \ref{11.29.C1} and Remark \ref{01.11.R1}, on the final time $T$, on $\alpha \in (0,1)$ in $A$,  and on the geometry of
  $\Omega_{\delta_0}$.
  \begin{theorem}\label{01.11.T2}
	Let $\al\in (0,1)$, $y_\de^0\in L^2(\Om_\de)$ and $f\in L^2(Q_\de)$. Then  the equation \eqref{01.11.14} with respect to $(y_\de^0,f_\de)$ has a unique weak solution
	\begin{equation*}
		y_\de\in L^2(0,T; H_0^1(\Om_\de;w_\de)), \mbox{ with } \pt_ty_\de\in L^2(0,T; H^{-1}(\Om_\de;w_\de))
	\end{equation*}
	 that satisfies the following estimate:
	\begin{equation}\label{01.11.16}
		\begin{split}
			&\esssup_{t\in [0,T]}\|y_\de(t)\|_{L^2(\Om_\de)}+\|y\|_{L^2(0,T;H_0^1(\Om_\de;w_\de))}+\|\pt_ty_\de\|_{L^2(0,T; H^{-1}(\Om_\de;w_\de))}\\
			&\leq C\left(\|f_\de\|_{L^2(Q_\de)}+\|y_\de^0\|_{L^2(\Om_\de)}\right),
		\end{split}
	\end{equation}
	where the constant $C>0$ depends only on $\al, N, T$ and $M$ and $\Om_{\de_0}$.
Moreover, if we further assume $y_\de^0\in H_0^1(\Om_\de;w_\de)$, then the weak solution
\begin{equation*}
	y_\de\in L^2(0,T; D(\mcA_\de))\cap L^2(0,T; H_0^1(\Om_\de;w_\de)), \mbox{ with } \pt_ty_\de\in L^2(Q_\de),
\end{equation*}
and it has  the estimate
\begin{equation}\label{01.11.17}
	\begin{split}
		&\esssup_{t\in [0,T]}\|y_\de(t)\|_{H_0^1(\Om_\de;w_\de)}+\|y_\de\|_{L^2(0,T; H^2(\mcO(\Ga^+,\de_0)))}+\|y_\de\|_{L^2(0,T; D(\mcA_\de))}+\|\pt_ty_\de\|_{L^2(Q_\de)}\\
		&\leq C\left(\|f_\de\|_{L^2(Q_\de)}+\|y_\de^0\|_{H_0^1(\Om_\de;w_\de)}\right),
	\end{split}
\end{equation}
where the constant $C>0$ depends only on $\al, N, T$, $M$, and $\Om_{\de_0}$.	
\end{theorem}

 The following corollary is equivalent to Corollary \ref{01.12.C1}.

\begin{corollary}\label{01.12.C2}
	Let $y_\de^0\in H_0^1(\Om_\de;w_\de)$ satisfy $\mcA_\de y_\de^0\in H_0^1(\Om_\de;w_\de)$. Let $y$ be the weak solution of \eqref{01.11.14} with respect to $(y_\de^0,0)$. Then,
	\begin{equation*}
		y_\de\in L^2(0, T; D(\mcA_\de^2)), \mbox{ with } \pt_t y_\de\in L^2(0, T; D(\mcA_\de)),
	\end{equation*}
	and
	\begin{equation*}
		\begin{split}
			\|\pt_ty_\de\|_{L^2(0, T; H^2(\mcO(\Ga^+,\de_0)))}\leq C\|\mcA y_\de^0\|_{H_0^1(\Om_\de;w_\de)},
		\end{split}
	\end{equation*}
	where the constant $C>0$ depends only on $\al, T, N$ and $M$ and $\Om_{\de_0}$.
\end{corollary}

 We now introduce a shape design framework for the degenerate parabolic equation \eqref{01.08.1} via the approximating equation \eqref{01.11.14}.

\begin{definition}\label{01.11.D2}
	For each $y_\de\in L^2(Q_\de)$, we denote
	\begin{equation*}
		Ey_\de=
		\begin{cases}
			y_\de, &\mbox{on } Q_\de, \\
			0, &\mbox{on } Q-Q_\de.
		\end{cases}
	\end{equation*}
	Then $Ey_\de\in L^2(Q)$ and $\|Ey_\de\|_{L^2(Q)}=\|y_\de\|_{L^2(Q_\de)}$. If $y_\de\in L^2(0,T; H_0^1(\Om_\de;w_\de))$, then
	\begin{equation*}
		Ey_\de\in L^2(0,T; H_0^1(\Om;w)), \mbox{ and } \|Ey_\de\|_{L^2(0,T; H_0^1(\Om;w))}=\|y_\de\|_{L^2(0,T; H_0^1(\Om_\de;w_\de))}.
	\end{equation*}
	If $y_\de\in H^1(0,T; L^2(\Om_\de))$, then
	\begin{equation*}
		Ey_\de\in H^1(0,T; L^2(\Om)), \pt_tEy_\de=E(\pt_ty_\de),\mbox{ and } \|Ey_\de\|_{H^1(0,T; L^2(\Om))}=\|y_\de\|_{H^1(0,T; L^2(\Om_\de))}.
	\end{equation*}
\end{definition}

 With the above preparations, we are now ready to prove one of the main results of this paper, namely Theorem \ref{01.11.T3}.

 \vspace{0.2cm}
\noindent {\bf {Proof of Theorem \ref{01.11.T3}.}}  For all $0<\de<\min\{\de_0, \dist(\supp y^0, \Ga_N^0)\}$, from $\Om_\de \s \Om$, $y^0|_{\Om_\de}=y^0\in H_0^1(\Om_\de;w_\de)$, \eqref{01.11.17}, and Definition \ref{01.11.D2}, we obtain
\begin{equation}\label{01.27.9}
	\begin{split}
		&\esssup_{t\in [0,T]}\|Ey_\de(t)\|_{H_0^1(\Om;w)}+\|\pt_tEy_\de\|_{L^2(Q)}\leq C\left(\|f\|_{L^2(Q)}+\|y^0\|_{H_0^1(\Om;w)}\right),
	\end{split}
\end{equation}
where the constant $C>0$ depends only on $\al, N, T$, $M$, and $\Om_{\de_0}$, but is independent of $\de\in (0,\min\{\de_0,\dist(\supp y^0, \Ga_N^0)\})$. This implies that there exists $z\in L^2(0,T; H_0^1(\Om;w))\cap H^1(0,T; L^2(\Om))$ such that
\begin{equation}\label{01.11.19}
	\begin{split}
		Ey_\de
		&\ra z \text{ weakly in } L^2(0,T;H_0^1(\Om;w)),\\
		\pt_t Ey_\de
		&\ra \pt_t z \text{ weakly in } L^2(Q).
	\end{split}
\end{equation}
Note that $y_\de$ is the weak solution of \eqref{01.11.4} with respect to $(y^0|_{\Om_\de}=y^0, f|_{Q_\de})$, hence for all $\psi\in C_0^\iy(Q)$, when $0<\de<\min\{\de_0, \dist(\supp\psi, Q)\}$, we have
\begin{equation*}
	-\iint_{Q_\de} y_\de \pt_t\psi\df x\df t+\iint_{Q_\de}  \nabla y_\de \cdot A_\de\nabla\psi\df x\df t=\iint_{Q_\de} f_\de \psi\df x\df t.
\end{equation*}
This shows that (see Definition \ref{01.11.D2})
\begin{equation*}
	-\iint_{Q} (Ey_\de) \pt_t\psi\df x\df t+\iint_{Q}  \nabla Ey_\de \cdot A\nabla\psi\df x\df t=\iint_{Q} f \psi\df x\df t.
\end{equation*}
Together with \eqref{01.11.19}, we obtain
\begin{equation*}
	-\iint_Q z\pt_t\psi\df x\df t+\iint_Q \nabla z\cdot A\nabla\psi\df x\df t=\iint_Q f\psi\df x\df t.
\end{equation*}
Hence $z$ satisfies
\begin{equation*}
	\pt_tz-\mcA z=f \text{ in the sense of distributions}.
\end{equation*}
From \eqref{01.11.19}, we get $\pt_tz\in L^2(Q)$, and since $f\in L^2(Q)$, we have
\begin{equation}\label{01.11.20}
	\pt_tz-\mcA z=f \text{ in } L^2(Q).
\end{equation}
Note that the embeddings $H_0^1(\Om;w)\hra L^2(\Om)$ and
\begin{equation*}
	\left\{u\in L^2(0,T; H_0^1(\Om;w))\colon \pt_tu\in L^2(Q)\right\}\hra L^2(Q)
\end{equation*}
are compact, from \eqref{01.27.9}, we get
\begin{equation}\label{01.11.21}
	Ey_\de\ra z\text{ strongly in } L^2(Q),
\end{equation}
and $z\in C([0,T]; L^2(\Om))$. From $y_\de(0)=y^0$, we obtain $z(0)=y^0$. Then from \eqref{01.11.19}, \eqref{01.11.20}, and \eqref{01.11.21} (or, \eqref{01.11.10}), we get the second conclusion of \eqref{01.11.18}, i.e., $z=y$.

Suppose, in addition, that $f = 0$. Similarly, from $\mcD(\Om)\s D(\mcA^2)$ (i.e., $y^0\in D(\mcA)$ and $\mcA y^0\in D(\mcA)$), Corollaries \ref{01.12.C1} and \ref{01.12.C2}, and $\Ga^+\s \pt\Om\cap\{x_N>3\de_0\}$, and the standard Sobolev trace theorem (see \eqref{01.11.8} and \eqref{01.11.17}, or Corollaries \ref{01.12.C1} and \ref{01.12.C2}), we get
\begin{equation*}
	\f{\pt y_\de}{\pt \nu}\in L^2(0,T; H^\f{1}{2}(\Ga^+)),\quad \pt_t\f{\pt y_\de}{\pt\nu}=\f{\pt(\pt_ty_\de)}{\pt\nu}\in L^2(0,T; H^\f{1}{2}(\Ga^+)),
\end{equation*}
and then the second conclusion in \eqref{01.12.5} holds by the embeddings $H^\f{1}{2}(\Ga^+)\hra L^2(\Ga^+)$
and
\begin{equation*}
	\left\{u\in L^2(0,T; H^\f{1}{2}(\Ga^+))\colon \pt_tu\in L^2(0,T; L^2(\Ga^+))\right\}\hra L^2(0,T; L^2(\Ga^+))
\end{equation*}
are compact.

It is clear that the first conclusion in \eqref{01.12.5} holds by \eqref{01.27.9} and the fact that the embedding $H_0^1(\Omega; w) \hookrightarrow L^2(\Omega)$ is compact.
We complete the proof of the theorem.
\hfill $\Box$

\section{Carleman estimates}\label{S4}

In this section, we establish a Carleman estimate for the approximating uniformly parabolic equations. Although this type of estimate is widely recognized, it is crucial to highlight that the constants involved must remain independent of the parameter
$\delta \in (0, \delta_0)$. Hence, we present a detailed proof.

We analyze the following backward equation:
\begin{equation}\label{01.24.1}
	\begin{cases}
		\pt_ty_\de+\Div(A\nabla y_\de)=f_\de, &\mbox{in }Q_\de, \\
		y_\de=0, &\mbox{on }\pt Q_\de,\\
		y_\de(T)=y_\de^T, &\mbox{in }\Om_\de,
	\end{cases}
\end{equation}
where $y_\de^T\in L^2(\Om_\de)$ and $f_\de\in L^2(Q_\de)$.

We denote
\begin{equation}\label{01.24.2}
	\eta=x_N^{2-\al},
\end{equation}
and
\begin{equation*}
	\Theta=\f{1}{[t(T-t)]^4}, \quad \xi=\Theta (\ga-\eta),
\end{equation*}
where the constant $\ga = |\eta|_{L^\iy(\Om)}+1$.  Then, we have
\begin{equation}\label{01.25.1}
	\nabla\eta=(2-\al)x_N^{1-\al}e_N, \quad \nabla\xi=-\Theta \nabla\eta=-(2-\al)\Theta x_{N}^{1-\al}e_N,
\end{equation}
and
\begin{equation}\label{01.25.2}
	\xi_t=\Theta'(\ga-\eta),\quad  \xi_{tt}=\Theta''(\ga-\eta),\quad \nabla\xi_t=-(2-\al)\Theta'x_N^{1-\al}e_N,
\end{equation}
and
\begin{equation}\label{01.27.1}
	\begin{split}
		|\Theta'|\leq C\Theta^\f{5}{4}, \quad |\Theta''|\leq C\Theta^\f{3}{2}, \quad |\Theta\Theta'|\leq C\Theta^\f{9}{4},
	\end{split}
\end{equation}
where the constants $C>0$ depends only on $T$.

Let us make the substitution
\begin{equation*}
	z=e^{-s\xi}y_\de\ \Lra \ y_\de =e^{s\xi}z.
\end{equation*}
 Subsequently, we obtain
\begin{equation}\label{01.24.3}
	z(t)=\nabla z(t)=0 \mbox{ at } t=0,T,
\end{equation}
and
\begin{equation}\label{01.24.4}
	z=\pt_tz=0 \mbox{ on }\pt Q_\de.
\end{equation}
 Moreover, equation \eqref{01.24.1} is equivalent to
\begin{equation}\label{01.28.2}
	e^{-s\xi}f_\de=P_1z+P_2z,
\end{equation}
where
\begin{equation}\label{com1}
	\left\{\begin{array}{l}
	\disp P_1z
	=\sum_{i=1}^3P_{1i}z=z_t+2sA\nabla z\cdot \nabla \xi+sz\Div(A\nabla\xi),\crr\disp
	P_2z
	=\sum_{i=1}^3P_{2i}z=\Div(A\nabla z)+sz\xi_t+s^2zA\nabla\xi\cdot\nabla\xi.
	\end{array}\right.
\end{equation}

\subsection{Computation}

 Given that equation \eqref{01.24.1} exhibits uniform parabolicity for every
 $\delta \in (0, \delta_0)$, the subsequent applications of integration by parts remain valid. Nevertheless, when
$\delta = 0$, the integration by parts operation concerning $(P_{12} z + P_{13} z,\, P_{21} z)_{L^2(Q_\delta)}$ in \dref{com1}
is no longer tenable, particularly in the vicinity of the set $\partial \Gamma_N^0$. For the sake of brevity in the subsequent discussion, we employ the notations $\int_{\Omega_\delta}$ and $\iint_{Q_\delta}$; in other instances, we omit the subscript
$\delta$. For instance, we use $w$ to denote  $w_\delta$, and the same applies to other variables.

{\it a) Compute $(P_{11}z, P_2z)_{L^2(Q_\de)}$.}

 Indeed, leveraging \eqref{01.25.2}, \eqref{01.24.3}, and \eqref{01.24.4}, we derive
\begin{equation*}
	\begin{split}
		(P_{11}z,P_2z)_{L^2(Q_\de)}
		&=-\f{s}{2}\iint_{Q_\delta} z^2\xi_{tt}\df x\df t-s^2\iint_{Q_\delta} z^2A\nabla\xi\cdot\nabla\xi_t\df x\df t\\
		&=-\f{s}{2}\iint_{Q_\delta} z^2\Theta''(\ga-\eta)\df x\df t-(2-\al)^2s^2\iint_{Q_\delta}\Theta \Theta'z^2x_N^{2-\al} \df x\df t.
	\end{split}
\end{equation*}

{\it b) Compute $(P_{12}z+P_{13}z,P_{21}z)_{L^2(Q_\de)}$.}

 Indeed, based on \eqref{01.24.4}, we obtain
\begin{equation*}
	\begin{split}
		&(P_{12}z+P_{13}z,P_{21}z)_{L^2(Q_\de)}\\
		&=2s\iint_{\pt Q_\delta}(A\nabla z\cdot\nu)(A\nabla z\cdot\nabla\xi)\df S\df t-s\iint_{\pt Q_\delta}(A\nabla z\cdot\nabla z)(A\nabla\xi\cdot \nu)\df S\df t\\
		&\hspace{4.5mm}-2s\sum_{i,j=1}^N\iint_{Q_\delta}A_i\f{\pt z}{\pt x_i}\f{\pt A_j}{\pt x_i}\f{\pt z}{\pt x_j}\f{\pt \xi}{\pt x_j}\df x\df t+s\sum_{i,j=1}^N\iint_{Q_\delta}\left(\f{\pt z}{\pt x_i}\right)^2\f{\pt A_i}{\pt x_j}A_j\f{\pt \xi}{\pt x_j}\df x\df t\\
		&\hspace{4.5mm}-2s\iint_{Q_\delta}A\nabla z\cdot \left[D^2\xi(A\nabla z)\right]\df x\df t-s\iint_{Q_\delta}zA\nabla z\cdot \nabla[\Div(A\nabla\xi)]\df x\df t,
		\end{split}
	\end{equation*}
	 and subsequently,
	\begin{equation*}
		\begin{split}
			(P_{12}z+P_{13}z,P_{21}z)_{L^2(Q_\de)}
		&=(2-\al)s\iint_{\pt Q_\de}\Theta x_N^{1-\al}|\nabla z|^2 (A\nu\cdot\nu)(-A\nu\cdot e_N)\df S\df t\\
		&\hspace{4.5mm}+(2-\al)^2 s\iint_{Q_\delta}\Theta x_N^\al (\pt_{x_N}z)^2\df x\df t.
	\end{split}
\end{equation*}

{\it c) Compute $(P_{12}z+P_{13}z,P_{22}z)_{L^2(Q_\de)}$.}

 Indeed, drawing upon equations \eqref{01.25.1}, \eqref{01.25.2}, and \eqref{01.24.4}, we obtain
\begin{equation*}
	\begin{split}
		(P_{12}z+P_{13}z,P_{22}z)_{L^2(Q_\de)}
		&=-s^2\iint_{Q_\de}z^2A\nabla\xi\cdot\nabla\xi_t\df x\df t\\
		&=-(2-\al)^2s^2\iint_{Q_\de}\Theta\Theta'x_N^{2-\al}z^2\df x\df t.
	\end{split}
\end{equation*}

{\it d) Compute $(P_{12}z+P_{13}z,P_{23}z)_{L^2(Q_\de)}$.}

 Indeed, based on equations \eqref{01.25.1} and \eqref{01.24.4}, we arrive at
\begin{equation*}
	\begin{split}
		(P_{12}z+P_{13}z,P_{23}z)_{L^2(Q_\de)}
		&=-s^3\iint_{Q_\de}z^2A\nabla\xi\cdot\nabla(A\nabla\xi\cdot\nabla\xi)\df x\df t\\
		&=(2-\al)^4s^3\iint_{Q_\de}\Theta^3z^2 x_N^{2-\al}\df x\df t.
	\end{split}
\end{equation*}

\subsection{Estimation}

Given the definition  $\ga= |\eta|_{L^\iy(\Om)}+1$, along with \eqref{01.27.1} and Lemma \ref{11.29.L1}, we obtain
\begin{equation*}
	\begin{split}
		-\f{s}{2}\iint_{Q_\de} z^2\Theta''(\ga-\eta)\df x\df t
		&\geq -Cs\ga \iint_{Q_\de}\Theta^\f{3}{2}z^2\df x\df t\\
		&\geq -Cs^2\ga^2\iint_{Q_\de}\Theta^2z^2x_N^{2-\al}\df x\df t-C\iint_{Q_\de}\Theta z^2x_N^{\al-2}\df x\df t\\
		&\geq -Cs^2\ga^2\iint_{Q_\de}\Theta^3z^2x_N^{2-\al}\df x\df t-\iint_{Q_\de}\Theta x_N^\al (\pt_{x_N}z)^2\df x\df t,
	\end{split}
\end{equation*}
where the constants $C>0$ depends only $\al$ and $T$ and $M$.   Furthermore, from \eqref{01.27.1}, we derive
\begin{equation*}
	\begin{split}
		-(2-\al)^2s^2\iint_{Q_\de}\Theta\Theta'x_N^{2-\al}z^2\df x\df t
		&\geq -Cs^2\iint_{Q_\de} \Theta^\f{9}{4} x_N^{2-\al}z^2\df x\df t\\
		&\geq -Cs^2\iint_{Q_\de}\Theta^3x_N^{2-\al}z^2\df x\df t,
	\end{split}
\end{equation*}
where the constants $C>0$  depends exclusively on
 $T$.
 These results imply the existence of $s_0\geq 1$, depending only on
$\al$ and $T$ and $|\eta|_{L^\iy(\Om)}$, such that for all $s\geq s_0$,  we have
\begin{equation*}
	\begin{split}
		(P_1z,P_2z)_{L^2(Q_\de)}
		&\geq (2-\al)s\iint_{\pt Q_\de}\Theta x_N^{1-\al}|\nabla z|^2 (A\nu\cdot\nu)(-A\nu\cdot e_N)\df S\df t\\
		&\hspace{4.5mm}+(2-\al)^2 s\iint_{Q_\delta}\Theta x_N^\al (\pt_{x_N}z)^2\df x\df t+(2-\al)^4s^3\iint_{Q_\de} \Theta^3z^2x_N^{2-\al}\df x\df t.
	\end{split}
\end{equation*}
 Consequently, there exists  $s_0\geq 1$,  depending only on $\al$ and $T$ and $|\eta|_{L^\iy(\Om)}$, such that for all
$s\geq s_0$,  we obtain
\begin{equation*}
	\begin{split}
		&s\iint_{Q_\delta}\Theta x_N^\al (\pt_{x_N}z)^2\df x\df t+s^3\iint_{Q_\de} \Theta^3z^2x_N^{2-\al}\df x\df t\\
		&\leq \|e^{-s\xi}f_\de\|_{L^2(Q_\de)}^2+Cs\iint_{\Ga^+\ts (0,T)}\Theta \left|\f{\pt z}{\pt \nu}\right|^2 \df S\df t,
	\end{split}
\end{equation*}
where the constant $C>0$ depends only on $\al$ and $T$ and $\Om_{\de_0}$.
 It is worth noting that there exist positive constants
$C$, depending only on $\al$ and $T$ and $\Om_{\de_0}$, such that (refer to Notation \ref{01.27.N1}).
\begin{equation*}
	C\left|\f{\pt z}{\pt \nu}\right|^2\leq \left|\f{\pt z}{\pt \nu_A}\right|^2\leq C\left| \f{\pt z}{\pt \nu}\right|^2 \mbox{ on }\Ga^+\ts (0,T), \mbox{ with } \f{\pt z}{\pt \nu_A}=A\nabla z\cdot \nu,
\end{equation*}
and $M$ indeed depends only on $\Om_{\de_0}$.
 We then arrive at the following Proposition \ref{01.27.P1}.
\begin{proposition}\label{01.27.P1}
	 Let  $\al\in (0,1)$, $\de\in (0,\de_0)$.  Suppose $y_\de^T\in L^2(\Om_\de)$ and $f_\de\in L^2(Q_\de)$ and $y_\de$
 denotes the weak solution of \eqref{01.24.1} corresponding to the pair $(y_\de^T,f_\de)$.
  Then, there exist two constants  $s_0\geq 1$ and $C>0$,  which depend only on
  $\al, T$ and $\Om_{\de_0}$, such that for all $s\geq s_0$, the following inequality holds:
	\begin{equation}\label{01.27.7}
		\begin{split}
			&s\iint_{Q_\delta}\Theta x_N^\al (\pt_{x_N}z)^2\df x\df t+s^3\iint_{Q_\de} \Theta^3z^2x_N^{2-\al}\df x\df t\\
			&\leq \|e^{-s\xi}f_\de\|_{L^2(Q_\de)}^2+Cs\iint_{\Ga^+\ts (0,T)}\Theta \left|\f{\pt z}{\pt \nu}\right|^2 \df S\df t.
		\end{split}
	\end{equation}
\end{proposition}

\section{Observability}\label{S5}

 Now, invoking Lemma \ref{01.11.L1}, we have
\begin{equation*}
	\begin{split}
		s\iint_{Q_\de}\Theta z^2\df x\df t
		&\leq Cs^2\iint_{Q_\de}\Theta z^2x_N^{2-\al}\df x\df t+C\iint_{Q_\de}\Theta z^2x_N^{\al-2}\df x\df t\\
		&\leq Cs^2\iint_{Q_\de}\Theta^3z^2x_N^{2-\al}\df x\df t+C\iint_{Q_\de} \Theta x_N^\al (\pt_{x_N}z)^2\df x\df t.
	\end{split}
\end{equation*}
Combining this with \eqref{01.27.7}, we arrive at
\begin{equation*}
	s\iint_{Q_\de}\Theta z^2\df x\df t\leq \|e^{-s\xi}f_\de\|_{L^2(Q_\de)}^2+Cs\iint_{\Ga^+\ts (0,T)}\Theta \left|\f{\pt z}{\pt \nu}\right|^2 \df S\df t
\end{equation*}
for $s\geq s_0$, that is,
\begin{equation} \label{01.27.8}
	\begin{split}
		s\iint_{Q_\de}\Theta y_\de^2 e^{-2s\xi}\df x\df t\leq \|e^{-s\xi}f_\de\|_{L^2(Q_\de)}^2+Cs\iint_{\Ga^+\ts (0,T)}\Theta  \left|\f{\pt y_\de}{\pt\nu}\right|^2e^{-2s\xi}\df x\df t
	\end{split}
\end{equation}
for $s\geq s_0$, where $s_0\geq 1$ and $C>0$ reply  only on $\al, T$ and $\Om_{\de_0}$ (note that $|\eta|_{L^\iy(\Om)}$ indeed depends only on $\Om_{\de_0}$).

Assume $f=0$.  Given that
\begin{equation*}
	\begin{split}
		\Theta e^{-2s_0\xi}\geq \Theta e^{-2s_0\ga\Theta}\geq C \mbox{ for } t\in \left(\f{T}{4},\f{3T}{4}\right),
	\end{split}
\end{equation*}
and
\begin{equation*}
	\Theta e^{-2s_0\xi}\leq \Theta e^{-2s_0\Theta}\leq C \mbox{ for } t\in (0,T),
\end{equation*}
where the constant $C>0$ depends only on $T$ and $s_0$, it follows that
\begin{equation*}
	\int_\f{T}{4}^\f{3T}{4}\int_{\Om_\de} y_\de^2\df x\df t\leq C\iint_{\Ga^+\ts (0,T)} \left|\f{\pt y_\de}{\pt\nu}\right|^2\df x\df t,
\end{equation*}
where the constant $C>0$ depends only on $\al, T$ and $\Om_{\de_0}$.  Observe that from
\begin{equation*}
	\begin{split}
		0=\f{1}{2}\f{\df}{\df t}\int_{\Om_\de} y_\de^2\df x-\int_{\Om_\de} A\nabla y_\de\cdot \nabla y_\de\df x
	\end{split}
\end{equation*}
we deduce
\begin{equation*}
	\|y_\de(t)\|_{L^2(\Om_\de)}<\|y_\de(s)\|_{L^2(\Om_\de)} \mbox{ for all } 0\leq t\leq s\leq T,
\end{equation*}
 and thus
\begin{equation}\label{01.27.10}
	\begin{split}
		\int_{\Om_\de}y_\de^2(0)\df x\leq \f{2}{T}\int_\f{T}{4}^\f{3T}{4}\int_{\Om_\de} y_\de^2\df x\df t\leq C\iint_{\Ga^+\ts(0,T)} \left|\f{\pt y_\de}{\pt\nu}\right|^2\df x\df t,
	\end{split}
\end{equation}
where the constant $C>0$ depends only on $\al, T$ and $\Om_{\de_0}$.

 By combining \eqref{01.27.10} and Theorem \ref{01.11.T3}, we arrive at the following Theorem \ref{01.27.T1}.

\begin{theorem}\label{01.27.T1}
Let $\al\in (0,1)$ and $y^T\in \mcD(\Om)$. Consider
$y$ as the backward weak solution of \eqref{01.08.1} corresponding to the initial-boundary conditions
$(y^0,f=0)$.  In other words, $y$ satisfies
	\begin{equation}\label{01.28.1}
		\begin{cases}
			\pt_{t}y-\Div (A\nabla y)= f, &\mbox{in } Q,\\
			y=0, &\mbox{on }\pt Q,\\
			y(0)=y^0, &\mbox{in } \Om.
		\end{cases}
	\end{equation}
Then, the following inequality holds:
	\begin{equation*}
		\int_{\Om}y^2(0)\df x\leq C\iint_{\Ga^+\ts (0,T)}\left|\f{\pt y}{\pt \nu}\right|^2\df S\df t,
	\end{equation*}
	where the constant $C>0$ depends only on $\al, T$ and $\Om_{\de_0}$.
\end{theorem}

\begin{proof} Consider  $y^T\in H_0^1(\Om;w)$. By the properties of this space, there exists a sequence
$\{y_n^T\}_{n\in\N}\s \mcD(\Om)$  such that  $y_n^T\ra y^T$ in $H_0^1(\Om;w)$ $n\ra\iy$. Applying Theorem \ref{01.27.T1}, we obtain
\begin{equation*}
	\int_\Om y_n^2(0)\df x\leq C\iint_{\Ga^+\ts (0,T)}\left|\f{\pt y_n}{\pt \nu}\right|^2\df S\df t.
\end{equation*}
 Furthermore, leveraging equation \eqref{01.11.8}, we derive the following inequalities
\begin{equation*}
	\iint_{\Ga^+\ts (0,T)}\left|\f{\pt (y_n-y)}{\pt \nu}\right|^2\df S\df t\leq C\|y_n-y\|_{L^2(0,T; H^2(\mcO(\Ga^+,\de_0)))}^2\leq C\|y_n^0-y^0\|_{H_0^1(\Om;w)}^2
\end{equation*}
and
\begin{equation*}
	\|y_n(0)-y(0)\|_{L^2(\Om)}\leq C\|y_n^0-y^0\|_{H_0^1(\Om;w)},
\end{equation*}
where the constant $C>0$ depends only on $\al, T$ and $\Om_{\de_0}$.   Consequently, we arrive at
\begin{equation}\label{01.27.11}
	\int_\Om y^2(0)\df x\leq C\iint_{\Ga^+\ts (0,T)}\left|\f{\pt y}{\pt\nu}\right|^2\df S\df t,
\end{equation}
 with the constant $C>0$  again depending only on  $\al, T$ and $\Om_{\de_0}$.

Next, let  $y^T\in L^2(\Om)$.  According to \eqref{01.11.7}, it follows that $y(\f{T}{2})\in H_0^1(\Om;w)$.
Utilizing \eqref{01.27.11}, we deduce
\begin{equation*}
	\begin{split}
		\int_\Om y^2(0)\df x\leq C\iint_{\Ga^+\ts (0,\f{T}{2})}\left|\f{\pt y}{\pt\nu}\right|^2\df S\df t\leq C\iint_{\Ga^+\ts (0,T)}\left|\f{\pt y}{\pt \nu}\right|^2\df S\df t.
	\end{split}
\end{equation*}
This completes the proof of the theorem.
\end{proof}

\end{document}